\documentclass[a4paper, bibliography=totoc]{scrartcl}
\usepackage{amsmath}
\usepackage{amssymb}
\usepackage{enumerate}
\usepackage{amsthm}
\usepackage{todonotes}
\usepackage{mathtools} 
\usepackage[UKenglish]{babel}
\usepackage[T1]{fontenc}
\usepackage[utf8]{inputenc}
\usepackage{tikz}
\usepackage{thmtools} 
\usepackage{thm-restate} 

\usepackage{color}
\definecolor{grey}{rgb}{.7,.7,.7}
\definecolor{blue}{rgb}{0,0,.8}
\definecolor{red}{rgb}{.8,0,0}
\definecolor{green}{rgb}{0,.4,0}
\definecolor{gold}{rgb}{0.8,0.6,0.1}
\definecolor{brown}{rgb}{0.8,0.4,0.1}

\def\md#1{\color{blue} {\ \textsc{MD: }}{ #1} \normalcolor}
\def\om#1{\color{red} {\ \textsc{OM: }}{ #1} \normalcolor}

\usepackage{hyperref} 
\definecolor{blue3}{rgb}{.1,.0,.4}
\hypersetup{colorlinks=true, linkcolor=red, urlcolor=blue3, citecolor=blue3, pdfpagemode=UseNone, pdfstartview=FitH, bookmarksopen=true} 
\usepackage[open]{bookmark} 

\declaretheorem[name=Theorem,numberwithin=section]{thm} 

\newtheorem*{define*}{Definition}
\newtheorem{define}[thm]{Definition}

\newtheorem*{lemma*}{Lemma}
\newtheorem{lemma}[define]{Lemma}

\newtheorem{corollary}[define]{Corollary}

\newtheorem*{algorithm*}{Algorithm}

\newtheorem*{remark*}{Remark}
\newtheorem{remark}[define]{Remark}

\newtheorem*{prop*}{Proposition}

\newtheorem*{obs*}{Observation}

\newtheorem*{fact*}{Fact}

\newtheorem*{quest*}{Question}

\newtheorem*{conjecture*}{Conjecture}

\newtheorem*{question*}{Question}

\numberwithin{claimcounter}{define}
\newtheorem*{claim*}{Claim}

\numberwithin{equation}{section}

\newcommand{\R}{\mathbb{R}}
\newcommand{\Q}{\mathbb{Q}}

\newcommand{\lip}{\operatorname{Lip}}

\newcommand{\dist}{\operatorname{dist}}
\newcommand{\N}{\mathbb{N}}

\newcommand{\id}{\operatorname{id}}

\DeclareMathOperator{\diam}{diam}

\newcommand{\mc}[1]{\mathcal{#1}}
\newcommand{\inter}{\operatorname{Int}}

\newcommand{\abs}[1]{\left|#1\right|}
\newcommand{\lnorm}[2]{\left\|#2\right\|_{#1}}
\newcommand{\norm}[1]{\left\|#1\right\|}
\newcommand{\opnorm}[1]{\lnorm{\text{op}}{#1}}

\newcommand{\set}[1]{\left\{#1\right\}}
\newcommand{\cl}[1]{\overline{#1}}

\newcommand{\Gat}{G\^ ateaux }

\newcommand{\lipone}{\lip_{1}}
\newcommand{\Ball}[1]{\mathbb{B}_{#1}}

\DeclarePairedDelimiter{\inner}{\langle}{\rangle}

\def\XXint#1#2#3{{\setbox0=\hbox{$#1{#2#3}{\int}$ }
		\vcenter{\hbox{$#2#3$ }}\kern-.6\wd0}}

\newcommand{\supnorm}[1]{\lnorm{\infty}{#1}}

\makeatletter
\newcommand{\mylabel}[2]{#2\def\@currentlabel{#2}\label{#1}}
\makeatother
\title{Typical Lipschitz mappings are typically non-differentiable.}
\author{Michael Dymond \and Olga Maleva}
\begin{document}
\maketitle
\begin{abstract}
	We prove that a typical Lipschitz mapping between any two Banach spaces is non-differentiable at typical points of any given subset of its domain in the most extreme form. This is a new result even for Lipschitz mappings between Euclidean spaces.
\end{abstract}

\section{Introduction}
The purpose of this note is to present a striking (non-)differentiablity property of typical Lipschitz mappings. Differentiability of Lipschitz mappings is the focus of mathematical research in an array of settings including Euclidean spaces (see e.g.~\cite{fitzpatrick1984differentiation}, \cite{Preiss_1990}, \cite{Dore_Maleva1}, \cite{preiss_speight2013}), Hilbert and Banach spaces (see e.g.~\cite{benyamini1998geometric}, \cite{LPT}, \cite{Dore_Maleva3}) and in geodesic metric spaces (see e.g.~\cite{Kirchheim94}, \cite{pinamonti_speight2017uds}). Moreover, the differentiability of typical Lipschitz mappings is specifically studied in the works \cite{preiss_tiser94}, \cite{Loewen_Wang_typical_2000}, \cite{dymond_maleva_2020}, \cite{dymond2019typical} and \cite{merlo}. However, despite the extent of research in this area, the property we prove in this paper appears to be missing from the existing literature. 

If we restrict our consideration to the real-valued case, we may recall that
the classical Rademacher's theorem guarantees that the set of non-differentiability points of 
a Lipschitz function defined on a finite-dimensional Euclidean space $\R^d$  is of
Lebesgue
measure zero,
and, more generally, a celebrated result by Preiss~\cite{Preiss_1990} says that any Lipschitz function defined on a Banach space $X$ with separable dual is differentiable on a dense subset of $X$.
As proof of the non-sharpness of Rademacher's Theorem, \cite[Corollary~6.5]{Preiss_1990} guarantees that for any $d\ge2$, there are null subsets $S$ of $\R^d$ (or even Hausdorff dimension one subsets of any Banach space with separable dual) which have the \textit{universal differentiability property}: any Lipschitz function defined on the space has points of differentiability in $S$. This result is further strengthened in~\cite{Dore_Maleva2,Dore_Maleva3,dymond_maleva2016} to show that null universal differentiability subsets (UDS) of $\R^d$ can be chosen to be compact, or at least closed and bounded of Hausdorff dimension $1$ in infinite-dimensional spaces with separable dual. The existence of UDS in spaces of dimension $d\ge2$ 	 contrasts with the fact that for any Lebesgue null subset $S$ of $\R$ there exists a Lipschitz 		   function whose set of non-differentiability points contains $S$, see~\cite{zahorski}.

On the other hand, in~\cite{preiss_tiser94} it is shown that given a null subset $S\subseteq\R$,
a \textit{typical} $1$-Lipschitz function $f\colon\R\to\R$ is nowhere differentiable on $S$  if and only if $S$ can be covered by a null $F_\sigma$ set. 
Moreover, if an analytic $S$ is not a subset of a null $F_\sigma$ set, a typical $1$-Lipschitz function has a point of differentiability in $S$.
In~\cite{dymond_maleva_2020} the authors of the present paper extend this dichtotomy of analytic subsets of $\R$ to all Euclidean spaces. In particular, \cite{dymond_maleva_2020} characterises those analytic subsets of $\R^{d}$ in which a typical Lipschitz function has points of differentiability.
The fact that there are closed null UDS in $\R^d$ if $d\ge2$ means that the existence of a null $F_\sigma$ cover needs to be replaced by a strictly stronger condition; in~\cite[Theorems~2.1~and~2.2]{dymond_maleva_2020} 
such a condition is given in terms of pure unrectifiability. 

According to \cite[Theorems~2.1 and 2.2]{dymond_maleva_2020}, 
the class of analytic sets in which a typical Lipschitz function finds points of differentiability is strictly larger than the class of UDS.
A finer analysis of such sets may consider the question of how many points of differentiability they provide for a typical Lipschitz function. However, in terms of category, the present work establishes that this set of differentiability points is always small. More precisely, we show that
for any given set $S\subseteq X$, 
a typical Lipschitz function is non-differentiable at a typical point of $S$.

Hence, this article further acts as an Erratum to~\cite[Remarks~2.9~and~3.18]{dymond_maleva_2020}. The authors hereby retract these two remarks, which are shown to be invalid by the present article. We further note that the errors referred to here are entirely contained within Remarks~3.18 and 2.9 in \cite{dymond_maleva_2020} and do not affect any of the results or indeed anything at all in the rest of the paper~\cite{dymond_maleva_2020}. 

\subsection{Main result.}\label{sec:typ}
Given a topological space $T$, we say that a typical element of $T$ possesses a certain property if the set of those elements of $T$ having that property is a residual subset of $T$. If $T$ is a complete metric space, then its residual subsets are dense in $T$, by the Baire Category theorem, hence a condition satisfied by a typical element is satisfied by elements of a dense $G_{\delta}$ subset.

It is therefore important to note that the term `typical' may have different meanings: the collection of typical objects depends on the ambient space, so when we say that a typical Lipschitz mapping satisfies a particular property, we need to clarify whether we consider the space of all Lipschitz mappings or only those with Lipschitz constant bounded by $L$ for certain $L>0$, whether the mappings in question are defined on the whole space or on its subset, and what topology (or metric) is used on the space of such Lipschitz mappings. Hence we need to take extra care as to what meaning we put into the term `a typical Lipschitz mapping'. We make this precise in subsection~\ref{sec:Lip}. 

Typical differentiability as an object of interest dates back to Banach's famous 1931 result \cite[Satz~1]{Banach1931} stating that a typical continuous function of an interval is nowhere differentiable. Inside the class of Lipschitz mappings between Euclidean spaces, Rademacher's theorem guarantees the existence of the derivative almost everywhere in the domain with respect to the Lebesgue measure. In the present article we prove that from a category point of view typical Lipschitz mappings remain nonetheless badly non-differentiable. More specifically, whilst Rademacher's theorem guarantees that the set of points of non-differentiability of any given Lipschitz mapping $\R^{d}\to\R^{l}$ has Lebesgue measure zero, we show that for the typical Lipschitz mapping, this set of non-differentiability points contains a $G_{\delta}$, dense subset of the domain. 
This result is known in the special case $l=1$, see~\cite[Theorem~4]{Loewen_Wang_typical_2000}; we cover all remaining Euclidean pairs $(\R^{d},\R^{l})$ as well as any more general Banach spaces. 
We point out that typical behaviour of real-valued Lipschitz mappings (the case $l=1$) has, in general, no direct implications for typical behaviour of vector-valued Lipschitz mappings (when $l>1$); see \cite[Theorem~6.1]{dymond_maleva_2020}.

We presently state the main result of this paper.
\begin{thm}\label{thm:typical_nonsep}
	Let $X$ and $Y$ be Banach spaces, $W$ be a separable subspace of $\mc{L}(X,Y)$, $Q$ be a closed and bounded subset of $X$ and let $G\subseteq \inter Q$. Then there is a residual subset $\mc{F}$ of $\lipone(Q,Y)$ such that for every $f\in \mc{F}$ the set
	\begin{equation*}
	\mc{N}_{f,W}:=\set{x\in G\colon \mc{D}_{f}(x)\supseteq \Ball{W}}
	\end{equation*}
	is residual in $G$.
\end{thm}
Here $\Ball{W}$ is the closed unit ball of $W$ and $\mc{D}_{f}(x)$ is the collection
of all those mappings in the space $\mc{L}(X,Y)$ of bounded linear operators $X\to Y$ which, in a specific sense, behave like a derivative of $f$ at $x$; for the precise definition of $\mc{D}_{f}(x)$ see \eqref{eq:Dfx} below. The inclusion $\mc{D}_{f}(x)\supseteq \Ball{W}$ for non-trivial $W$ and points $x\in \mc{N}_{f,W}$ in Theorem~\ref{thm:typical_nonsep} implies, in particular, that $\mc{N}_{f,W}$ is contained in the set of G\^ateaux non-differentiability points of $f$. However, this condition should be interpreted as a very strong form of non-differentiability. We elaborate on this presently.

A particularly strong form of non-differentiability of a function at a point, studied in \cite[Theorem~1.9]{maleva_preiss2018}, occurs when many different linear mappings simultaneously behave like a derivative of the function. Note that if $f$ is G\^ateaux differentiable then either $\mc{D}_{f}(x)$ is empty, or it is the singleton set containing only the G\^ateaux derivative of $f$; if $f$ is 
Fr\'echet differentiable at $x$ we have $\mc{D}_{f}(x)=\set{Df(x)}$.
Accordingly, the size of the set $\mc{D}_{f}(x)$ provides a measure of the severity of non-differentiability of $f$ at $x$. In the case that $f$ is $1$-Lipschitz and $\mc{L}(X,Y)$ is separable (which is satisfied, for example, when $X$ and $Y$ are finite-dimensional), the most extreme form of non-differentiability of $f$ at $x$ occurs when $\mc{D}_{f}(x)$
equals the closed unit ball $\Ball{\mc{L}(X,Y)}$ of $\mc{L}(X,Y)$.

When $\mc{L}(X,Y)$ is non-separable the notion of strongest form of non-differentiability according to the size of $\mc{D}_{f}(x)$ is less clear. This is because, as we show in Lemma~\ref{lemma:Df_separable}, $\mc{D}_{f}(x)$ is always separable. Thus, in particular, if $\mc{L}(X,Y)$ is non-separable, then it is impossible to achieve $\mc{D}_{f}(x)=\Ball{\mc{L}(X,Y)}$. Qualitatively, the strongest form of non-differentiability of a $1$-Lipschitz $f$ that may occur in such case is $\mc{D}_{f}(x)=\Ball{W}$ for an infinite-dimensional, separable subspace of $\mc{L}(X,Y)$. In both the separable and non-separable case, the non-differentiability of typical Lipschitz mappings that we establish in our main result is of the strongest possible form described above.

In relation to Theorem~\ref{thm:typical_nonsep}, which applies to any pair $(X,Y)$ of Banach spaces, only partial information for a narrow subclass of Banach space pairs has been available. 
In fact, even for the class of Euclidean pairs $(X,Y)=(\R^{d},\R^{l})$, there are no comparable statements to Theorem~\ref{thm:typical_nonsep}.

For the special case where $X$ is finite dimensional and $Y=\R$ a variant of the non-differentiability of a typical Lipschitz function on a residual subset of the space is given by~\cite[Theorem~4]{Loewen_Wang_typical_2000}. 
	However, even in the settings of~\cite[Theorem~4]{Loewen_Wang_typical_2000} the result proved in this paper is stronger.
The examination, by Theorem~\ref{thm:typical_nonsep}, of the size of the set of non-differentiability points of a typical Lipschitz mapping inside an arbitrary subset $G$ of the domain is unprecedented in the literature. In particular,~\cite{Loewen_Wang_typical_2000}
requires $G$ to be equal to the whole space. This aspect of Theorem~\ref{thm:typical_nonsep} is crucial in order to refute the remarks of \cite{dymond_maleva_2020} discussed above. 
In the next corollary, we distinguish this one particular application of Theorem~\ref{thm:typical_nonsep}, although it applies of course to many other settings. The statement of Corollary~\ref{cor} directly refutes the Remarks~2.9 and~3.18 of \cite{dymond_maleva_2020}, as announced above.
\begin{corollary}\label{cor}
	Let $d\ge1$ and $S\subseteq(0,1)^d$ be arbitrary.
	Then there is a residual subset $\mc{F}$ of $\lipone([0,1]^d,\R)$ such that for every $f\in \mc{F}$ the set
	of non-differentiability points of $f$ in $S$ is residual in $S$.	
\end{corollary}
\begin{proof}
	Apply Theorem~\ref{thm:typical_nonsep} to $X=\R^d$, $Y=\R$, $W=\mc{L}(\R^d,\R)$, $Q=[0,1]^d$ and $G=S$.
\end{proof}
Moreover, we note that Theorem~\ref{thm:typical_nonsep} further establishes a strong form of non-differentiability, stronger in particular than the emptiness of the Dini subgradient obtained by \cite[Theorem~4]{Loewen_Wang_typical_2000} in the restricted setting; see Lemma~\ref{lemma:Dini}. 
	
\section{Preliminaries and Notation.}
\subsection{General notation and differentiability notions.}
Given a normed vector space $X$, we let $\Ball{X}$ denote its closed unit ball and $\mathbb{S}_X$ its unit sphere. An open ball in $X$ with centre $x$ and radius $r$ will be written as $B_{X}(x,r)$ and for closed balls we write $\cl{B}_{X}$ instead of $B_{X}$. The origin in $X$ will be denoted by $0_{X}$. If $Y$ is an additional normed vector space, we let $\mc{L}(X,Y)$ denote the space of bounded linear operators $X\to Y$. The operator norm on $\mc{L}(X,Y)$ is denoted by $\opnorm{-}$. For a subset $E$ of a topological space, we let $\inter E$ denote the interior of $E$. A subset $\Gamma$ of a metric space $(M,d)$ will be called \emph{uniformly separated} if $\inf\set{d(x,y)\colon x,y\in\Gamma,\,x\neq y}>0$. For such a set $\Gamma$ and $s>0$, we call $\Gamma$ \emph{$s$-separated} if $\inf\set{d(x,y)\colon x,y\in\Gamma,\,x\neq y}\geq s$.

For a mapping $f\colon Q\subseteq X\to Y$ and $x\in \inter Q$ we let 
\begin{equation}\label{eq:Dfx}
\mc{D}_{f}(x):=\set{L\in \mc{L}(X,Y)\colon \liminf_{r\to 0+}\sup_{u\in\cl{B}_{X}(0_{X},r)}\frac{\lnorm{Y}{f(x+u)-f(x)-Lu}}
	{r}=0}.	
\end{equation}
Observe that if $f$ is $1$-Lipschitz, we have $\mc{D}_{f}(x)\subseteq \Ball{\mc{L}(X,Y)}$ for every $x\in \inter Q$. Moreover, the set $\mc{D}_{f}(x)$ and the standard notions of differentiability of $f$ at $x$ only make sense when $x$ belongs to the interior of the domain of $f$. We will never consider such notions for boundary points of the domain.

\subsection{Optimality of Theorem~\ref{thm:typical_nonsep}.}
Taking $W$ in Theorem~\ref{thm:typical_nonsep} as any non-trivial, separable subspace of $\mc{L}(X,Y)$ ensures that the set $\mc{N}_{f,W}$ is contained in the set of points of \Gat non-differentiability of $f$. Moreover, when $\mc{L}(X,Y)$ is itself separable, we may take $W=\mc{L}(X,Y)$. In this case we get that for a typical $f\in \lipone(Q,Y)$ we have that the set $\mc{D}_{f}(x)$ is maximal, as it is equal to $\Ball{W}=\Ball{\mc{L}(X,Y)}$, at a typical point $x$ of $G$. Moreover, in the following lemma we show that it is not possible to omit the separability condition on $W$ in Theorem~\ref{thm:typical_nonsep}. 

\begin{lemma}\label{lemma:Df_separable}
	Let $X$ and $Y$ be normed spaces, $Q\subseteq X$, $x\in \inter Q$ and $f\colon Q\to Y$ be a mapping. Then the set $\mc{D}_{f}(x)$ is separable.
\end{lemma}
\begin{proof}
	Let $r>0$ be small enough so that $B_{X}(x,r)\subseteq Q$. For each rational $q\in \Q\cap (0,r)$ and each $n\in\N$ choose $L_{q,n}\in\Ball{\mc{L}(X,Y)}$ such that
	\begin{multline*}
	\sup_{u\in \cl{B}_{X}(0_{X},q)}\frac{\lnorm{Y}{f(x+u)-f(x)-L_{q,n}u}}{q}\\
	\leq \inf_{L\in\mc{L}(X,Y)}\sup_{u\in \cl{B}_{X}(0_{X},q)}\frac{\lnorm{Y}{f(x+u)-f(x)-Lu}}{q}+\frac{1}{n}.
	\end{multline*}
	We show that the set $\set{L_{q,n}\colon q\in\Q\cap (0,r),\,n\in\N}$ is dense in $\mc{D}_{f}(x)$. 
	
	Indeed, consider arbitrary $L_0\in \mc{D}_{f}(x)$ and $\varepsilon>0$. Choose $q\in\Q\cap (0,r)$ so that
	\begin{equation*}
	\sup_{u\in\cl{B}_{X}(0,q)}\frac{\lnorm{Y}{f(x+u)-f(x)-L_0u}}{q}\leq \frac{\varepsilon}{3}
	\end{equation*}
	Next, fix any $n\in \N$ such that $n\geq \frac{3}{\varepsilon}$. Then, for every $u\in \cl{B}_{X}(0,q)$ we have
	\begin{align*}
	\lnorm{Y}{(L_{q,n}-L_0)u}&\leq \lnorm{Y}{L_{q,n}u+f(x)-f(x+u)}+\lnorm{Y}{f(x+u)-f(x)-L_0u}\\
	&\leq \frac{\varepsilon q}{3}+\frac{q}{n}+\frac{\varepsilon q}{3}\leq \varepsilon q,
	\end{align*}
	which implies $\opnorm{L_{q,n}-L_0}\leq \varepsilon$.
\end{proof}

Let us record a simple comparison, in the case of real-valued $f$, of the set $\mc{D}_{f}(x)$ with the Dini subgradient $\hat{\partial}f(x)$ of $f$ at $x$, considered in \cite{Loewen_Wang_typical_2000}. Let $f\colon X\to \R$ be a function, $x\in X$ and for each $v\in X$ consider the lower Dini directional derivative
	\begin{equation*}
	f_{+}(x;v):=\liminf_{t\to 0+}\frac{f(x+tv)-f(x)}{t}.
	\end{equation*}
	The Dini subgradient of $f$ at $x$ is then defined by
	\begin{equation*}
	\hat{\partial}f(x):=\set{x^{*}\in X^{*}\colon f_{+}(x;v)\geq \inner{x^{*},v}\,\forall v\in X}.
	\end{equation*}
\begin{lemma}\label{lemma:Dini}
		Let $X$ be a normed space, $f\colon X\to \R$ be a $1$-Lipschitz function, $x,v\in X$ and $y^ {*},z^{*}\in \mc{D}_{f}(x)$ 
be such that
$z^{*}(v)<0<y^{*}(v)$. Then $\hat{\partial}f(x)=\emptyset$.
	\end{lemma}
	\begin{proof}
		Observe that $z^{*}\in\mc{D}_{f}(x)$ and $z^{*}(v)<0$ implies $f_{+}(x;v)\leq z^{*}(v)<0$. Similarly $y^{*}\in\mc{D}_{f}(x)$ and $y^{*}(-v)<0$ implies $f_{+}(x;-v)<0$. Thus, we have $f_{+}(x;v)<0$ and $f_{+}(x;-v)<0$, which implies $\hat{\partial}f(x)=\emptyset$.
\end{proof}
\begin{remark}

	Note that, by the Hahn-Banach theorem, Lemma~\ref{lemma:Dini} may be applied whenever $\mc{D}_{f}(x)\supseteq \mathbb{S}_{X^{*}}$, the unit sphere of $X^*$. In this situation, given any $v\in X\setminus\set{0_X}$, we may choose the functional $y^{*}\in \mathbb{S}_{X^{*}}$ such that $\lnorm{X^{*}}{y^*}=\inner{y^{*},\frac{v}{\lnorm{X}{v}}}=1$ and then set $z^{*}=-y^{*}$.
\newline
It can be easily seen from the example of $f(x)=-\norm{x}\colon X\to\R$ that a $1$-Lipschitz function may have
$\mc{D}_f(0)=\hat{\partial}f(0)=\emptyset$. Hence, there is no reverse implication to the statement of Lemma~\ref{lemma:Dini}: emptiness of the Dini subgradient $\hat{\partial}f(x)$ does not imply any type of largeness of the set $\mc{D}_{f}(x)$.
\end{remark}

\subsection{Lipschitz mappings.}\label{sec:Lip}
Given metric spaces $(Q,d_{Q})$, $(Y,d_{Y})$, we denote by $\lipone(Q,Y)$ the set of Lipschitz mappings $f\colon Q\to Y$ with $\lip(f)\leq 1$.  

If $Q$ is a bounded metric space, then $\lipone(Q,Y)$ becomes a metric space when equipped with the metric $\rho\colon\lipone(Q,Y)\times\lipone(Q,Y)\to[0,\infty)$ defined by
\begin{equation*}
\rho(f,g)=\sup\set{d_{Y}(f(x),g(x))\colon x\in Q}.
\end{equation*}
Moreover, if both $Q$ and $Y$ are complete, then the metric space $(\lipone(Q,Y),\rho)$ is also complete. If $Q$ is compact and $Y$ is separable then the metric space $(\lipone(Q,Y),\rho)$ is separable, see~\cite[Thm~4.19]{kechris2012classical}. 

In what follows, we will only consider the case where $Q$ is a closed and bounded subset of a Banach space $X$ and $Y$ is a Banach space. In this case the metric $\rho$ is given by 
\begin{equation*}
\rho(f,g)=\rho_Q(f,g)=\supnorm{f-g}=\sup\set{\lnorm{Y}{g(x)-f(x)}\colon x\in Q},
\end{equation*}
and makes $(\lipone(Q,Y),\rho_Q)$ a complete metric space. The latter is needed, as mentioned in subsection~\ref{sec:typ},
in order for residual subsets of $\lipone(Q,Y)$ to be dense in $\lipone(Q,Y)$.

Note that if $Q\subseteq X$ is not bounded, one could still consider the space $\lipone(Q,Y)$ 
as a complete metric space with metric
\[
\tilde\rho(f,g)=\sum_{n=1}^{\infty}2^{-n}\min\set{1,\rho_{Q\cap n\Ball{X}}(f,g)}. 
\]
This is the approach chosen in \cite{Loewen_Wang_typical_2000}. In the present work, we elect to work only with bounded $Q$ in order to be consistent with the papers \cite{preiss_tiser94} and \cite{dymond_maleva_2020}. However, we note that
the proofs given in the present paper may be easily 
modified to obtain the same results for $\lipone(Q,Y)$ in the case $Q$ is unbounded.
\subsection{The Banach-Mazur Game}\label{sec:BMgame}

To prove that a set is residual, i.e.\ a complement of the set of first Baire category, we will make use of the Banach-Mazur game.

Given a topological space $T$ and its subset $H\subseteq T$, the Banach-Mazur game in $T$ with target $H$ is played by two players, Player~I and Player~II, as follows: The game starts by Player~I selecting a non-empty open subset $U_{1}$ of $T$. Player~II must then respond by nominating a non-empty open subset $V_{1}$ of $T$ with $V_{1}\subseteq U_{1}$. In the $k$-th round of the game, with $k\geq 2$, Player~I chooses a non-empty open set $U_{k}\subseteq V_{k-1}$ and Player~II returns a non-empty open set $V_{k}\subseteq U_{k}$. Thus, a run of the game is described by an infinite sequence of open sets
\begin{equation*}
U_{1}\supseteq V_{1}\supseteq U_{2}\supseteq V_{2}\supseteq \ldots\supseteq U_{k}\supseteq V_{k}\supseteq\ldots,
\end{equation*}
where the sets $U_{k}$ are the choices of Player~I and the sets $V_{k}$ are those of Player~II. Player~II wins the game if 
\begin{equation*}
\bigcap_{k\in\N}V_{k}\subseteq H.
\end{equation*}
Otherwise Player~I wins.

The Banach-Mazur game can be used to determine whether a subset of a topological space is residual. More precisely, for any non-empty topological space $T$ and any subset $H$ of $T$ it holds that $H$ is a residual subset of $T$ if and only if Player~II has a winning strategy in the Banach Mazur game in $T$ with target set $H$; see \cite[Thm~8.33]{kechris2012classical}.

In the case that $T$ is a metric space (as will be the case in our setting), open balls may be used in place of the open sets $U_{k}$ and $V_{k}$ above, see also~\cite[Theorem~3.16]{dymond_maleva_2020}. Thus, the moves of Player~I and Player~II effectively become a choice of a pair $(x,r)$ where $x\in T$ prescibes the centre of the ball and $r>0$ the radius. 

\section{Proof of Theorem~\ref{thm:typical_nonsep}}
In this section we prove Theorem~\ref{thm:typical_nonsep}. For the proof of subsequent auxiliary lemmata we follow the convention that the infimum of the empty set is $+\infty$. We also note that in any normed space a bounded non-empty subset has a non-empty boundary.

The next lemma is a generalisation of \cite[Lemma~3.1]{mynewpaper} for normed spaces instead of convex sets.
\begin{lemma}\label{lemma:general}
	Let $X$ and $Z$  be normed spaces, $0<a<b$, and let $f_1,f_2\colon X\to Z$ be Lipschitz mappings such that $\lip(f_1)+\lip(f_2)\le1$ and $f_1(0_{X})=f_2(0_{X})=0_{Z}$. 
	Then there exists a mapping $\Phi=\Phi(a,b,f_1,f_2)\colon X\to Z$ such that
	\begin{enumerate}[(i)]
		\item\label{Phi1a} $\Phi(x)=f_1(x)$ for all $x\in  \cl{B}_{X}(0_{X},a)$.
		\item\label{Phi2a} $\Phi(x)=f_2(x)$ for all $x\in X\setminus B_{X}(0_{X},b)$.
		\item\label{Phi3a} $\lip(\Phi)\leq 1+\frac{a}{b-a}$.
		\item \label{Phi5}
		If $f_1=0_{Z}$ is the constant $0_{Z}$ mapping, then 
		$\lnorm{Z}{\Phi(x)-f_2(x)}\le a\lip(f_2)$ for all $x\in X$. 
		\item\label{Phi4a}
		If $f_2=0_{Z}$ is the constant $0_{Z}$ mapping, then
		$\lnorm{Z}{\Phi(x)}\leq b\lip(f_1)$ for all $x\in X$.
	\end{enumerate}
	
\end{lemma}
\begin{proof}
	Define $\Phi\colon X\to Z$ by 
	\begin{equation*}
		\Phi(x)=\begin{cases}
			f_1(x) & \text{ if }x\in  \cl{B}_{X}(0_{X},a),\\
			\displaystyle\frac{b-\lnorm{X}{x}}{b-a}f_1(x)+
			\frac{b\bigl(\lnorm{X}{x}-a\bigr)}{\lnorm{X}{x}(b-a)}f_2(x)
			& \text{ if }x\in B_{X}(0_{X},b)\setminus \cl{B}_{X}(0_{X},a),\\
			f_2(x) & \text{ if } x\in X\setminus B_{X}(0_{X},b).
		\end{cases}
	\end{equation*}
	Clearly, $\Phi$ satisfies~\eqref{Phi1a} and~\eqref{Phi2a}.
	Observe that $\Phi$ is a continuous mapping $X\to Z$. 
	Moreover, since $f_1,f_2\in\lipone(X,Z)$, in order to check~\eqref{Phi3a}, it is enough to verify 
	\begin{equation}\label{eq:Lip_bound}
		\lnorm{Z}{\Phi(y)-\Phi(x)}\le {\left(1+\frac{a}{b-a}\right)}\lnorm{X}{y-x}
	\end{equation}
	for any $x,y\in B_{X}(0_{X},b)\setminus \cl{B}_{X}(0_{X},a)$.
	To show this inequality, fix $x,y\in B_{X}(0_{X},b)\setminus \cl{B}_{X}(0_{X},a)$ and note first that
	\begin{multline}\label{eq.Phi-b}
		\lnorm{Z}{\Phi(y)-\Phi(x)}\\
		\\
		\le
		\lnorm{Z}{
			\frac{b-\lnorm{X}{y}}{b-a}f_1(y)-
			\frac{b-\lnorm{X}{x}}{b-a}f_1(x)
		}
		+
		\lnorm{Z}{
			\frac{b\bigl(\lnorm{X}{y}-a\bigr)}{\lnorm{X}{y}(b-a)}f_2(y)
			-
			\frac{b\bigl(\lnorm{X}{x}-a\bigr)}{\lnorm{X}{x}(b-a)}f_2(x)
		}.
	\end{multline}	
	The first term of~\eqref{eq.Phi-b} is  bounded above by 
	\begin{align*}	
		&\frac{b-\lnorm{X}{y}}{b-a}
		\lnorm{Z}{
			f_1(y)-f_1(x)}
		+
		\abs{\frac{b-\lnorm{X}{y}}{b-a}-\frac{b-\lnorm{X}{x}}{b-a}}
		\lnorm{Z}
		{f_1(x)}
		\\
		\\
		\le
		&\frac{b-\lnorm{X}{y}}{b-a}
		\lip(f_1)\lnorm{X}{y-x}
		+
		\frac{\lnorm{X}{y-x}}{b-a}\lip(f_1)\lnorm{X}{x}
		\\
		\\
		=
		&\frac{b-\lnorm{X}{y}+\lnorm{X}{x}}{b-a}\lip(f_1)\lnorm{X}{y-x}
		\\
		\\
		\le
		&\left(1+\frac{a}{b-a}\right)\lip(f_1)\lnorm{X}{y-x}
		,
	\end{align*}	
	as we may assume, without loss of generality, that $\lnorm{X}{y}\ge\lnorm{X}{x}$. We also note that in the first inequality we used  $f_1(0_X)=0_Z$ to deduce $\lnorm{Z}{f_{1}(x)}\le\lip(f_1)\lnorm{X}{x}$.
	
	The second term of~\eqref{eq.Phi-b} is bounded above by
	\begin{align*}
		&\frac{b\bigl(\lnorm{X}{y}-a\bigr)}{\lnorm{X}{y}(b-a)}
		\lnorm{Z}{f_2(y)-f_2(x)}+
		\abs{\frac{b\bigl(\lnorm{X}{y}-a\bigr)}{\lnorm{X}{y}(b-a)}-
			\frac{b\bigl(\lnorm{X}{x}-a\bigr)}{\lnorm{X}{x}(b-a)}
		}\lnorm{Z}{f_2(x)}\\
	\\
		\le 
		&\frac{b\bigl(\lnorm{X}{y}-a\bigr)}{\lnorm{X}{y}(b-a)}
		\lip(f_2)\lnorm{X}{y-x}
		+\frac{ab}{\lnorm{X}{x}\lnorm{X}{y}(b-a)}\lnorm{X}{y-x}\lip(f_2)\lnorm{X}{x}\\
		\\
		=
		&\left(1+\frac{a}{b-a}\right)\lip(f_2)\lnorm{X}{y-x},
	\end{align*}
	where we again used $\lnorm{Z}{f_2(x)}\le\lip(f_2)\lnorm{X}{x}$ due to $f_2(0_X)=0_Z$.
	
	Summing the derived upper bounds for the two terms of \eqref{eq.Phi-b} establishes \eqref{eq:Lip_bound}, so~\eqref{Phi3a} follows.
	
	Finally, 
	if $f_1=0_{Z}$, then for any $x\in B_X(0_X,b)\setminus \overline{B}_X(0_X,a)$
	\[
	\lnorm{Z}{\Phi(x)-f_2(x)}
	=
	\frac{a\bigl|\lnorm{X}{x}-b\bigr|}{(b-a)\lnorm{X}{x}}\lnorm{Z}{f_2(x)}
	\le a\lip(f_2),
	\]
	for any $x\in \cl B_X(0_X,a)$
	\[
	\lnorm{Z}{\Phi(x)-f_2(x)}
	=
	\lnorm{Z}{f_2(x)}
	\le\lip(f_2)\lnorm{X}{x}
	\le
	a\lip(f_2)
	\]
	and for any $x\in X\setminus B_X(0_X,b)$
	\[
	\lnorm{Z}{\Phi(x)-f_2(x)}
	=0,
	\]
	which verifies~\eqref{Phi5}.

	If, however,
	$f_2=0_{Z}$, then for any $x\in B_X(0_X,b)\setminus \overline{B}_X(0_X,a)$ 
	\[
	\lnorm{Z}{\Phi(x)}=
	\frac{b-\lnorm{X}{x}}{b-a}\lnorm{Z}{f_1(x)}
	\le
	\lip(f_1)\lnorm{X}{x}
	\le b\lip(f_1)
	,
	\]
	and for 
	any $x\in \overline{B}_X(0_X,a)$
	\[
	\lnorm{Z}{\Phi(x)}=
	\lnorm{Z}{f_1(x)}\le 
	a\lip(f_1)\le b\lip(f_1).
	\]
	This verifies~\eqref{Phi4a} as the inequality trivially holds for $x\in X\setminus B_X(0,b)$.
\end{proof}
The following lemma provides a construction which will be used to define a winning strategy for Player~II in the Banach-Mazur game in Lemma~\ref{lemma:BMgame1}. 
The property~\eqref{eq:g_formula} of $g$ ensures that this new $1$-Lipschitz mapping ``sees'' $L$ as its derivative in a small neighbourhood of the given set $\Gamma$.
\begin{lemma}\label{lemma:PlayerII_nonsep-1}
	Let $X$ and $Y$ be normed spaces and $Q\subseteq X$ be a bounded, closed set with $\inter Q\neq \emptyset$.  Let $r\in (0,1)$, $L\in\mc{L}(X,Y)$ with $\opnorm{L}\leq 1-r$ and $f\in \lipone(Q,Y)$.  Let $\emptyset\ne\Gamma\subseteq \inter Q$ be a uniformly separated set with 
	\begin{equation}\label{eq:Gamma-cond-1}
		\inf_{x\in\Gamma}\dist_{X}(x,\partial Q)>0. 
	\end{equation}
	Then there exist $\alpha\in(0,r)$ and $g\in\lipone(Q,Y)$ such that $\supnorm{g-f}<r$ and
	\begin{equation}\label{eq:g_formula}
		g(x+u)=g(x)+Lu\qquad\text{for all }u\in \cl{B}_{X}(0_{X},\alpha)\text{ and all }x\in\Gamma.
	\end{equation}
\end{lemma}
\begin{proof}
	The approach we take to modify the mapping $f$ to arrive at $g$ is similar to that taken in \cite[Lemma~3.3]{mynewpaper}.
	
	The conclusion of this lemma is valid for $f$ if and only if it is valid for any mapping of the form $f+p$, where $p\colon Q\to Y$ is a constant mapping. Therefore, we may assume that $0_{Y}\in f(\Gamma)$. Lipschitz mappings $h\colon Q\to Y$ with the property $0_{Y}\in h(Q)$ satisfy $\lnorm{Y}{h(x)}\leq \lip(h)\diam Q$ for all $x\in Q$. This fact will be used later in the proof.
	
	Fix $s\in (0,1)$ small enough so that $\Gamma$ is $4s$-separated and the infimum of \eqref{eq:Gamma-cond-1} is at least $4s$. Let 
	\begin{equation}\label{eq:beta0}
		\beta=\beta(r,s,Q)\in (0,s/2)
	\end{equation}
	be a parameter depending only on $r, s$ and $Q$ which will be determined at the end of the proof in \eqref{eq:alpha_beta}. We define first a mapping $g_{0}\colon Q\to Y$ by
	\begin{equation*}
		g_{0}(z)=\begin{cases}
			f(z) & \text{ if }z\in Q\setminus \bigcup_{x\in\Gamma}B_{X}(x,s),\\
			f(x+\Phi(z-x)) & \text{ if }z\in B_{X}(x,s)\text{ and }x\in\Gamma,
		\end{cases}
	\end{equation*}
	where $\Phi:=\Phi(\beta,s,0_{X},\id_X)\colon X\to X$ is the mapping given by Lemma~\ref{lemma:general} applied to $X$, $Z=X$, $a=\beta$, $b=s$, $f_{1}=0_{X}$ (the constant mapping $X\to X$ with value $0_{X}$) and $f_{2}=\id_{X}$. 
	Using again $\beta<s$ and Lemma~\ref{lemma:general}~\eqref{Phi1a}, we note that
	\begin{equation}\label{eq:g0_const}
		g_{0}(z)=g_{0}(x)=f(x)\qquad \text{whenever }
		x\in \Gamma\text{ and }z\in \cl{B}_{X}(x,\beta).
	\end{equation}
	It follows, in particular, that
	\[
	0_{Y}\in f(\Gamma)=g_{0}(\Gamma).
	\]
	Further, we note that Lemma~\ref{lemma:general}~\eqref{Phi3a} and~\eqref{Phi5} imply
	\begin{equation*}
		\lip(g_{0})\leq 1+\frac{\beta}{s-\beta}
		\qquad 
		\text{and}
		\qquad 
		\supnorm{g_{0}-f}\leq \beta.
	\end{equation*}

	Let $T=T(r,s,Q,L)\in\mc{L}(X,Y)$ be a linear operator 
	with $\opnorm{T}\leq 1$. This operator will be used in the construction of the target $1$-Lipschitz mapping $g$ such that~\eqref{eq:g_formula} is satisfied with a multiple of $T$ instead of $L$; this will determine how $T$ is defined, see~\eqref{eq:choice_T}. The choice of $T$ depends on $L,s$ and $\beta=\beta(r,s,Q)$, and we note that $L,Q,r,s$ are fixed from the start.

	For each $x\in \Gamma$ we define a mapping $h_{x}\colon X\to Y$ by
	\begin{equation*}
		h_{x}(z)=g_{0}(x)+T(z-x),\qquad z\in X.
	\end{equation*}
	Note that $\lip(h_{x})= \opnorm{T}\leq 1$ for every $x\in \Gamma$. Next we let $\alpha=\alpha(r,s,Q)\in (0,\beta)$ be a further parameter to be determined later in the proof in \eqref{eq:alpha_beta} and define $g_{1}\colon Q\to Y$ by
	\begin{equation*}
		g_{1}(z)=\begin{cases}
			g_{0}(z), & \text{ if }z\in Q\setminus \bigcup_{x\in\Gamma}B_{X}(x,\beta),\\
			h_{x}(x+\Psi(z-x)), & \text{ if }z\in B_{X}(x,\beta)\text{ and }x\in \Gamma,
		\end{cases}
	\end{equation*}
	where $\Psi:=\Phi(\alpha,\beta,\id_{X},0_{X})\colon X\to X$ is the mapping given by Lemma~\ref{lemma:general} applied to $X$, $Z=X$, $a=\alpha$, $b=\beta$, $f_{1}=\id_{X}$ and $f_{2}=0_{X}$. The properties of $\Psi$ and \eqref{eq:g0_const} ensure that $g_{1}$ is continuous. With the continuity of $g_{1}$ established, we may estimate its Lipschitz constant as
	\begin{multline*}
		\lip(g_{1})\leq \max\set{\lip(g_{0}),\left(1+\frac{\alpha}{\beta-\alpha}\right)}\\
		\leq\max\set{\left(1+\frac{\beta}{s-\beta}\right), \left(1+\frac{\alpha}{\beta-\alpha}\right)}\leq 1+\frac{\beta}{s-\beta}
		=\frac{s}{s-\beta},
	\end{multline*}
	where the penultimate inequality is achieved by imposing the condition
	\begin{equation}\label{eq:alpha_condition}
		\alpha\leq \frac{\beta^{2}}{s}.
	\end{equation} 
	Moreover, we have 
	\begin{equation*}
		\supnorm{g_{1}-g_{0}}\leq \beta\qquad\text{and}\qquad 0_{Y}\in g_{0}(\Gamma)=g_{1}(\Gamma).
	\end{equation*}
	To verify the former, observe, using \eqref{eq:g0_const}, $\opnorm{T}\leq 1$, $\lip(\text{id}_X)=1$ and Lemma~\ref{lemma:general}~\eqref{Phi4a}, that for any $z\in B_{X}(x,\beta)$ with $x\in \Gamma$ 
	\begin{multline*}
		\lnorm{Y}{g_1(z)-g_0(z)}
		=\lnorm{Y}{g_0(x)+T(\Psi(z-x))-g_0(z)}\\
		=\lnorm{Y}{T(\Psi(z-x))}
		\leq \lnorm{X}{\Psi(z-x)}
		\leq  \beta.
	\end{multline*}
	Finally we set 
	\begin{equation*}
		g_{2}=\frac{s-\beta}{s}\cdot g_{1},
	\end{equation*}
	so that $g_{2}\in\lipone (Q,Y)$ and
	\begin{equation*}
		\supnorm{g_{2}-g_{1}}\leq \frac{\beta \lip(g_{1})\diam Q}{s}\leq \frac{\beta \diam Q}{s-\beta}\leq \frac{2\beta \diam Q}{s},
	\end{equation*}
	using $\beta<s/2$ from~\eqref{eq:beta0}.
	Setting $g:=g_{2}\in \lipone(Q,Y)$, we conclude that
	\begin{equation*}
		\supnorm{g-f}\leq \supnorm{g_{2}-g_{1}}+\supnorm{g_{1}-g_{0}}+\supnorm{g_{0}-f}
		\leq \frac{2\beta\diam Q}{s}+2\beta\leq \frac{4\beta(1+\diam Q)}{s}.
	\end{equation*}
	Thus, we achieve $\supnorm{g-f}\leq r$ by imposing the condition
	\begin{equation}\label{eq:beta_condition1}
		\beta \leq \frac{rs}{4(1+\diam Q)}.
	\end{equation}
	We are now ready to make the choice of linear operator $T=T(r,s,Q,L)\in \mc{L}(X,Y)$ with $\opnorm{T}\leq 1$. Indeed, the choice 
	\begin{equation}\label{eq:choice_T}
		T=\frac{s}{s-\beta}L, 
	\end{equation}
	establishes~\eqref{eq:g_formula}. 
	We note that the condition \eqref{eq:beta_condition1} imposed on $\beta$ implies $\beta\leq rs$, which together with $\opnorm{L}\leq 1-r$ gives $\opnorm{T}\leq 1$. 
	
	It only remains to note that the choices
	\begin{equation}\label{eq:alpha_beta}
		\beta =\frac{rs}{4(1+\diam Q)},\qquad \alpha=\frac{r^{2}s}{16(1+\diam Q)^{2}}
	\end{equation}
	satisfy the required conditions~\eqref{eq:beta0}, \eqref{eq:alpha_condition} and~\eqref{eq:beta_condition1}.
\end{proof}

\begin{lemma}\label{lemma:Gamma1}
	Let $X$ be a normed space, $Q$ be a bounded subset of $X$ and $G\subseteq \inter Q$. Then there exists a sequence $(\Gamma_{k})_{k\in\N}$ of nested sets $\Gamma_k\subseteq \Gamma_{k+1}\subseteq G$ such that the union $\bigcup_{k\ge 1}\Gamma_k$ is dense in $G$ and each set $\Gamma_{k}$ satisfies the hypothesis of Lemma~\ref{lemma:PlayerII_nonsep-1}, that is, 
	$\Gamma_k$ satisfies~\eqref{eq:Gamma-cond-1} and is $\delta_k$-separated for some $\delta_k>0$.
\end{lemma}
\begin{proof}
	If $G=\emptyset$, let $\Gamma_k=\emptyset$ for all $k\in\mathbb N$.
	
	Assume $G\ne\emptyset$.
	Let $G_k=\set{x\in G\colon\dist_{X}(x,\partial Q)\geq 2^{-k}}$. 
	Since $G\subseteq\inter Q$, we have that $\bigcup_{k\ge1}G_k=G$. Let $n\ge1$ be the smallest index such that $G_n\ne\emptyset$.  Set $\Gamma_k=\emptyset$ for any $0\le k\le n-1$. For any $k\ge n$, let us make an inductive choice of $\Gamma_k\supseteq \Gamma_{k-1}$ to
	be a non-empty maximal $2^{-k}$-separated subset of $G_k$. 
	Since for any $k\ge n$ the set $\Gamma_k\ne\emptyset$ is a $2^{-k}$-net of $G_k$,  we  conclude that $\bigcup_{k\ge n}\Gamma_k=\bigcup_{k\ge 1}\Gamma_k$ is dense in $G$.	
\end{proof}
The following lemma is the final step allowing us to prove Theorem~\ref{thm:typical_nonsep}. It shows that every bounded linear operator $L$ with $\opnorm{L}<1$ behaves like a derivative of a typical $\lip_1$ function, at a typical point of $G$.
\begin{lemma}\label{lemma:BMgame1}
	Let $X$ and $Y$ be Banach spaces, $Q$ be a closed and bounded subset of $X$ with non-empty interior, $G\subseteq \inter Q$ and $L\in \mc{L}(X,Y)$ with $\opnorm{L}<1$. Then there is a residual subset $\mc{H}_{L}$ of $\lipone(Q,Y)$ such that for every $f\in \mc{H}_{L}$ the set
	\begin{equation*}
		\mc{P}_{L,f}:=\set{x\in G\colon L\in \mc{D}_{f}(x)}
	\end{equation*}
	is residual in $G$.
\end{lemma}
\begin{proof}
	We assume that $G\ne \emptyset$.	
	Let 
	\begin{equation*}
		\mc{H}_{L}:=\set{f\in\lipone(Q,Y)\colon \text{ the set }{\mc{P}_{L,f}}\text{ is residual in $G$}}.
	\end{equation*}
	We prove that the set $\mc{H}_{L}$ is residual in $\lipone(Q,Y)$ by describing a winning strategy for Player~II in the relevant Banach-Mazur game in $\lipone(Q,Y)$ with the target $\mc{H}_L$,
	in which Player I’s choices are balls $B(f_k,r_k)$ and Player II’s choices
	are balls $B(g_k,s_k)$; see subsection~\ref{sec:BMgame} for details on the Banach-Mazur game.  Here and throughout the proof, given a mapping $\phi\in\lipone(Q,Y)$ and $\rho>0$ we abbreviate the notation $B_{\lipone(Q,Y)}(\phi,\rho)$, for the open ball in the space $\lipone(Q,Y)$ with centre $\phi$ and radius $\rho$, to $B(\phi,\rho)$. 
	
	Before the game starts, let Player~II prepare by fixing a nested sequence $(\Gamma_{k})_{k\in \N}$ of sets $\Gamma_k\subseteq \Gamma_{k+1}\subseteq G$, given by Lemma~\ref{lemma:Gamma1}. 
	
	Let $k\geq 1$ and let $f_{k}\in\lip_{1}(Q,Y)$ and $r_{k}>0$ denote the $k$-th move of Player~I. If $k\ge2$, then 
	\begin{equation*}
		{B}(f_{k},r_{k})\subseteq B(g_{k-1},s_{k-1}).
	\end{equation*}
	Since Player~II may always replace the radius $r_{k}$
	by any smaller radius $\tilde r_{k}>0$  
	we may assume that  
	\begin{equation}\label{eq:xk-fk}
		r_{k}\leq 2^{-k}(1-\opnorm{L})
		\quad\text{and}\quad 
		\cl{B}(f_{k+1},r_{k+1})\subseteq B(g_{k},s_{k})
		\text{ for every }
		k\geq 1.
	\end{equation}
	
	Player~II then responds by applying Lemma~\ref{lemma:PlayerII_nonsep-1} to find a mapping $g_{k}\in \lipone(Q,Y)$ and $\alpha_{k}\in(0,r_k)$ satisfying $\supnorm{g_k-f_k}<r_{k}$ and
	\begin{equation}\label{eq:gk_formula-1}
		g_{k}(x+u)=
		g_k(x)+Lu,\qquad  \text{whenever } 
		x\in\Gamma_k\,\text{ and }u\in\cl{B}_{X}(0_{X},\alpha_{k}).\\
	\end{equation}
	Finally, Player~II chooses 
	\begin{equation}\label{eq:sk1}
		0<s_{k}<\alpha_k/k
	\end{equation}
	small enough so that 
	\begin{equation}\label{eq:sk2}
		\cl{B}(g_{k},s_{k})\subseteq B(f_{k},r_{k}) 
	\end{equation}
	and declares $g_{k}\in\lipone(Q,Y)$ and $s_{k}>0$ as their $k$-th move.
	
	Due to the conditions~\eqref{eq:xk-fk}  and~\eqref{eq:sk2} the intersection
	\begin{equation*}
		\bigcap_{k=1}^{\infty}B(f_{k},r_{k})=\bigcap_{k=1}^{\infty}B(g_{k},s_{k})
	\end{equation*}
	is a singleton set containing only the Lipschitz mapping $g:=\lim_{k\to \infty}g_{k}\in\lipone(Q,Y)$.
	
	We define a sequence $(U_{k})_{k\in\N}$ of open sets $U_{k}\subseteq X$ by
	\begin{equation*}
		U_{k}:=\bigcup_{x\in\Gamma_{k}} B_{X}(x,s_{k})
	\end{equation*}
	and consider the set $J\subseteq G$ given by
	\begin{equation*}
		J:=G\cap \bigcap_{n=1}^{\infty}\bigcup_{k=n}^{\infty}U_{k}.
	\end{equation*}
	The set $J$ is clearly a relatively $G_{\delta}$ subset of $G$. Moreover, $\bigcup_{k\ge n}U_{k}\supseteq \bigcup_{k\ge n}\Gamma_{k}=\bigcup_{k\ge 1}\Gamma_{k}$, as $\Gamma_k$ are nested, and the latter is a dense subset of $G$ by Lemma~\ref{lemma:Gamma1}; thus $J\supseteq \bigcup_{k\ge 1}\Gamma_{k}$ is dense in $G$.
	We conclude that $J$ is a relatively residual subset of $G$.

	To complete the proof, we show that $L\in \mc{D}_{g}(x)$ for every $x\in J$. Let $x\in J$ and $\varepsilon>0$. Choose $k\in\N$ with $k\geq {4}/{\varepsilon}$ such that $x\in U_{k}$ and $\alpha_{k}<\varepsilon$. Let $x_{k}\in\Gamma_{k}$ be such that $x\in B_{X}(x_{k},s_{k})$; let  $u\in\cl{B}_{X}(0_{X},\alpha_{k})$ be arbitrary. Then, applying~\eqref{eq:gk_formula-1}, we get
	$g_{k}(x_{k}+u)=g_{k}(x_{k})+Lu$. 
	Using this identity, we derive
	\begin{multline*}
		\lnorm{Y}{g(x+u)-g(x)-Lu}\leq\lnorm{Y}{g(x+u)-g_{k}(x+u)}+\lnorm{Y}{g_{k}(x+u)-g_{k}(x_{k}+u)}\\
		+\lnorm{Y}{g_{k}(x_{k}+u)-g_{k}(x_{k})-Lu}+\lnorm{Y}{g_{k}(x_{k})-g_{k}(x)}+\lnorm{Y}{g_{k}(x)-g(x)}\\
		\leq  2\supnorm{g_{k}-g}+2\lnorm{X}{x_{k}-x}\leq 4
		s_k\leq \frac{4\alpha_{k}}{k},
	\end{multline*}	
	where the last inequality is due to Player~II's choice \eqref{eq:sk1} of $s_{k}$. This argument verifies
	\[
	\sup_{u\in\cl{B}_{X}(0,\alpha_{k})}\frac{\lnorm{Y}{g(x + u)-g(x)-Lu}}{\abs{\alpha_k}}
	\le\frac4k\leq \varepsilon
	.
	\]
	and subsequently $L\in\mc{D}_{g}(x)$. Since $x\in J$ was arbitrary, we conclude that the residual subset $J$ of $G$ is contained in $\mc{P}_{L,g}$. Hence, $g\in \mc{H}_{L}$ and Player~II wins the game.
\end{proof}
We are now ready to prove Theorem~\ref{thm:typical_nonsep}.
\begin{proof}[Proof of Theorem~\ref{thm:typical_nonsep}]
	Fix a dense sequence $(L_{n})_{n\in\N}$ in the closed unit ball $\Ball{W}$ with $\opnorm{L_{n}}<1$ for all $n\in\N$. 
	Let the sets $\mc{H}_{L_{n}}\subseteq \lipone(Q,Y)$ be given by the conclusion of Lemma~\ref{lemma:BMgame1}.
	Define a residual subset $\mc{F}$ of $\lipone(Q,Y)$ by
	\begin{equation*}
		\mc{F}:=\bigcap_{n\in\N}\mc{H}_{L_{n}}
	\end{equation*}
	and let $f\in\mc{F}$ be arbitrary.
	Let the sets $\mc{P}_{L_{n},f}$ be given by the conclusion of Lemma~\ref{lemma:BMgame1}. 
	Consider the residual subset $\mc{P}=\bigcap_{n\in\N}\mc{P}_{L_{n},f}$ of $G$
	and let $x\in\mc{P}$.	
	To complete the proof, we verify that $\mc{D}_{f}(x)\supseteq \Ball{W}$.  
	
	Let $L\in \Ball{W}$ be arbitrary and then extract a subsequence of $(L_{n})_{n\in\N}$ which converges to $L$. To simplify the notation, we will denote this subsequence also by $(L_{n})_{n\in\N}$. For each $n\in\N$ we have $x\in\mc{P}_{L_n,f}$ and so we may choose  $0<t_{n}\leq 2^{-n}$ such that $\sup_{u\in\cl{B}_X(0_{X},t_{n})}\lnorm{Y}{f(x+u)-f(x)-L_{n}u}\leq 2^{-n}{t_{n}}$. Then for every $n\in\N$ and $u\in\cl{B}_X(0_{X},t_{n})$ we have
	\begin{multline*}
		\lnorm{Y}{f(x+u)-f(x)-Lu}\leq \lnorm{Y}{f(x+u)-f(x)-L_{n}u}+\lnorm{Y}{L_{n}u-Lu}\\
		\leq (2^{-n}+\opnorm{L_{n}-L}){t_{n}}.
	\end{multline*}
	Thus, $L\in\mc{D}_{f}(x)$.
\end{proof}
\newpage
\bibliographystyle{plain}
\bibliography{biblio}

\begin{thebibliography}{10}

\bibitem{Banach1931}
S.~Banach.
\newblock {Über die Baire'sche Kategorie gewisser Funktionenmengen}.
\newblock {\em Studia Mathematica}, 3(1):174--179, 1931.

\bibitem{benyamini1998geometric}
Y.~Benyamini and J.~Lindenstrauss.
\newblock {\em Geometric nonlinear functional analysis}, volume~48.
\newblock American Mathematical Soc., 1998.

\bibitem{Dore_Maleva1}
M.~Dor{\'e} and O.~Maleva.
\newblock {A compact null set containing a differentiability point of every
  Lipschitz function}.
\newblock {\em Mathematische Annalen}, 351(3):633--663, Nov 2011.

\bibitem{Dore_Maleva2}
M.~Dor{\'e} and O.~Maleva.
\newblock {A compact universal differentiability set with Hausdorff dimension
  one}.
\newblock {\em Israel Journal of Mathematics}, 191(2):889--900, Oct 2012.

\bibitem{Dore_Maleva3}
M.~Doré and O.~Maleva.
\newblock {A universal differentiability set in Banach spaces with separable
  dual}.
\newblock {\em Journal of Functional Analysis}, 261(6):1674 -- 1710, 2011.

\bibitem{dymond2019typical}
M.~Dymond.
\newblock {Typical differentiability within an exceptionally small set}.
\newblock {\em Journal of Mathematical Analysis and Applications},
  490(2):124317, 2020.

\bibitem{mynewpaper}
M.~Dymond.
\newblock {Porosity phenomena of non-expansive, Banach space mappings}.
\newblock {\em arXiv preprint arXiv:2110.13722}, 2021.

\bibitem{dymond_maleva2016}
M.~Dymond and O.~Maleva.
\newblock {Differentiability inside sets with Minkowski dimension one}.
\newblock {\em Michigan Math. J.}, 65(3):613--636, 08 2016.

\bibitem{dymond_maleva_2020}
M.~Dymond and O.~Maleva.
\newblock A dichotomy of sets via typical differentiability.
\newblock {\em Forum of Mathematics, Sigma}, 8:e41, 2020.

\bibitem{fitzpatrick1984differentiation}
S.~Fitzpatrick.
\newblock {Differentiation of real-valued functions and continuity of metric
  projections}.
\newblock {\em Proceedings of the American Mathematical Society},
  91(4):544--548, 1984.

\bibitem{kechris2012classical}
A.~Kechris.
\newblock {\em Classical descriptive set theory}, volume 156.
\newblock Springer Science \& Business Media, 2012.

\bibitem{Kirchheim94}
B.~Kirchheim.
\newblock {Rectifiable Metric Spaces: Local Structure and Regularity of the
  Hausdorff Measure}.
\newblock {\em Proceedings of the American Mathematical Society},
  121(1):113--123, 1994.

\bibitem{LPT}
J.~Lindenstrauss, D.~Preiss, and J.~Tiser.
\newblock {\em Fréchet differentiability of Lipschitz functions and porous
  sets in Banach spaces}.
\newblock 01 2012.

\bibitem{Loewen_Wang_typical_2000}
P.~D. Loewen and X.~Wang.
\newblock {Typical Properties of Lipschitz Functions}.
\newblock {\em Real Analysis Exchange}, 26(2):717 -- 726, 2000.

\bibitem{maleva_preiss2018}
O.~Maleva and D.~Preiss.
\newblock {Cone unrectifiable sets and non-differentiability of Lipschitz
  functions}.
\newblock {\em Israel Journal of Mathematics}, 8 2018.

\bibitem{merlo}
A.~Merlo.
\newblock {Full non-differentiability of typical Lipschitz functions}.
\newblock {\em arXiv:1906.08366}, 2019.

\bibitem{pinamonti_speight2017uds}
A.~Pinamonti and G.~Speight.
\newblock {A measure zero universal differentiability set in the Heisenberg
  group}.
\newblock {\em Mathematische Annalen}, 368(1-2):233--278, 2017.

\bibitem{Preiss_1990}
D~Preiss.
\newblock {Differentiability of Lipschitz functions on Banach spaces}.
\newblock {\em Journal of Functional Analysis}, 91(2):312 -- 345, 1990.

\bibitem{preiss_speight2013}
D.~Preiss and G.~Speight.
\newblock {Differentiability of Lipschitz Functions in Lebesgue Null Sets}.
\newblock {\em Inventiones mathematicae}, 197, 2013.

\bibitem{preiss_tiser94}
D.~Preiss and J.~Tišer.
\newblock {Points of non-differentiability of typical Lipschitz functions}.
\newblock {\em Real Analysis Exchange}, 20(1):219--226, 1994.

\bibitem{zahorski}
Z.~Zahorski.
\newblock Sur les ensembles des points de divergence de certaines
  int\'{e}grales singuli\`eres.
\newblock {\em Ann. Soc. Polon. Math.}, 19:66--105 (1947), 1946.

\end{thebibliography}

\noindent Michael Dymond\\Mathematisches Institut,
Universität Leipzig,
PF 10 09 02,
04109 Leipzig,
Deutschland.\\
\texttt{michael.dymond@math.uni-leipzig.de}

\noindent Olga Maleva\\
School of Mathematics,
University of Birmingham,
Birmingham,
B15 2TT,
United Kingdom.\\
\texttt{O.Maleva@bham.ac.uk}\\[3mm]

\end{document}